\documentclass[10 pt]{amsart}

\usepackage{amssymb,amsmath,amsfonts,mathrsfs}
\usepackage[title]{appendix}

\newtheorem{theorem}{Theorem}[section]
\newtheorem{lemma}{Lemma}[section]
\newtheorem{corollary}{Corollary}[section]
\newtheorem{prop}{Proposition}[section]

\newtheorem{definition}{Definition}[section]

\theoremstyle{remark}
\newtheorem{remark}{Remark}[section]

\numberwithin{equation}{section}

\newcommand{\dom}{\operatorname{Dom}}

\newcommand\CC{\mathbb{C}}
\newcommand\RR{\mathbb{R}}
\newcommand\ZZ{\mathbb{Z}}
\newcommand\NN{\mathbb{N}}
\newcommand\TT{\mathbb{T}}
\newcommand\del{\Delta}
\newcommand\p{\partial}

\newcommand\sml{S^{m}(\ZZ^n\times \TT^n)}

\begin{document}

\title{Extended Sobolev Scale on $\mathbb{Z}^n$}
\author{Ognjen Milatovic}
\address{Department of Mathematics and Statistics\\
         University of North Florida   \\
       Jacksonville, FL 32224 \\
        USA
           }
\email{omilatov@unf.edu}



\begin{abstract}
In analogy with the definition of ``extended Sobolev scale" on $\RR^n$ by Mikhailets and Murach, working in the setting of the lattice $\ZZ^n$, we define the ``extended Sobolev scale" $H^{\varphi}(\ZZ^n)$, where $\varphi$ is a function which is $RO$-varying at infinity. Using the scale $H^{\varphi}(\ZZ^n)$, we describe all Hilbert function-spaces that serve as interpolation spaces with respect to a pair of discrete Sobolev spaces $[H^{(s_0)}(\ZZ^n), H^{(s_1)}(\ZZ^n)]$, with $s_0<s_1$. We use this interpolation result to obtain the mapping property and the Fredholmness property of (discrete) pseudo-differential operators (PDOs) in the context of the scale $H^{\varphi}(\ZZ^n)$. Furthermore, starting from a first-order positive-definite (discrete) PDO $A$ of elliptic type, we define the ``extended discrete $A$-scale" $H^{\varphi}_{A}(\ZZ^n)$ and show that it coincides, up to norm equivalence, with the scale $H^{\varphi}(\ZZ^n)$. Additionally, we establish the $\ZZ^n$-analogues of several other properties of the scale $H^{\varphi}(\RR^n)$.
\end{abstract}

\maketitle

\section{Introduction}\label{S:S-intro}

Owing to their role in the discretization of continuous problems, difference equations and the corresponding pseudo-differential operators (PDO) on the lattice $\ZZ^n$ have attracted quite a bit of attention in the last decade; see, for instance the papers~\cite{CDK-20, CK-19, DW-13, gjbnm-16, Mol-10, Rab-10, Rab-9, Rod-11}, keeping in mind that in some articles $n=1$. Another milestone in this field of study was marked by the article~\cite{BRK-20}, in which the authors developed the corresponding global symbol calculus. In this setting, $\mathbb{Z}^n\times\TT^n$ (here, $\TT^n:=\RR^n/\ZZ^n$)  plays the role of the phase space, whereby the frequency component belongs to the $n$-torus $\TT^n$. It turns out that the symbol class $S^{m}_{\rho,\delta}(\ZZ^n\times\TT^n)$, with $m,\rho,\delta\in\RR$, introduced in~\cite{BRK-20} (see definition~\ref{D-1} below for the special case $S^{m}(\ZZ^n\times\TT^n)$ with $\rho=1$ and $\delta=0$), bears some resemblance to that of the so-called $SG$-operators discussed in~\cite{Das-08,DW-08}. With the symbol calculus at their disposal, the authors of~\cite{BRK-20} proved a number of fundamental results concerning, among other things, $\ell^2(\ZZ^n)$-boundedness (here, $\ell^2(\ZZ^n)$ denotes the space of square summable complex-valued functions on $\ZZ^n$), compactness, and $H^{(s)}(\ZZ^n)$-boundedness of the corresponding operators, where, $H^{(s)}(\ZZ^n)$, $s\in\RR$, indicates (see section~\ref{SS:s-2-6} below for details) the Sobolev scale on $\ZZ^n$. Subsequently, the authors of~\cite{DK-20} studied further properties of PDO on $\ZZ^n$, including the relationship between the maximal and minimal realizations of a pseudo-differential operator in $\ell^2(\ZZ^n)$. More recently, the authors of~\cite{km-21} considered weighted $M^{m}_{\rho,\Lambda}(\ZZ^n\times\mathbb{T}^n)$-type symbols (discrete counterparts of $M^{m}_{\rho,\Lambda}$-type symbols on $\RR^n$ discussed in~\cite{DM-23,GM-03,GM-05, W-06}), and after developing the corresponding calculus and defining the corresponding weighted Sobolev scale on $\ZZ^n$, investigated various questions concerning the corresponding PDOs on $\ZZ^n$.

In parallel with the developments described in the preceding paragraph, starting with the seminal papers~\cite{MM-09,MM-13} (see also section 2.4 of the monograph~\cite{MM-14}), Mikhailets and Murach proposed the so-called extended Sobolev scale $H^{\varphi}(\RR^n)$, defined similar to $H^{(s)}(\RR^n)$, $s\in\RR$, but with $\varphi(\langle\xi\rangle)$ in place of $\langle\xi\rangle^{s}$, where $\varphi$ is a function $RO$-varying at $\infty$ and satisfies some additional properties; see section~\ref{SS-1-8} below for precise description of the function class $RO$. It turns out that the scale $H^{\varphi}(\RR^n)$ is a more general version of the so-called refined Sobolev scale, introduced by Mikhailets and Murach earlier in~\cite{MM-07} (see also section 1.3 in~\cite{MM-14}). What makes the scale $H^{\varphi}(\RR^n)$ particularly interesting is the following interpolation property (established in~\cite{MM-13}; see also~\cite{MM-15} for bounded domains $\Omega\subset\RR^n$ with Lipschitz boundary): a Hilbert space $\mathscr{S}$ is an interpolation space with respect to a pair (see section~\ref{SS-2-1} below for this concept) of the form
\begin{equation*}
[H^{(s_0)}(\RR^n),H^{(s_1)}(\RR^n)], \qquad -\infty<s_0<s_1<\infty,
\end{equation*}
if and only if $\mathscr{S}=H^{\varphi}(\RR^n)$, for some $\varphi\in RO$.

The aforementioned interpolation property (and other useful attributes studied in~\cite{MM-07,MM-13, MM-15}) of the refined (and extended) Sobolev spaces, sets a convenient stage for elaborating the theory of elliptic boundary-value problems on $\RR^n$ (and closed manifolds) and for establishing various results from spectral theory of differential operators on $\RR^n$ (and closed manifolds) in analogy with those that hold in the setting of the usual Sobolev spaces.  To get a taste of the degree of activity in this area during the last fifteen years, besides the papers mentioned so far, we refer the reader to the monograph~\cite{MM-14}, papers~\cite{DMM-21,AK-16,MM-21,Z-17}, and numerous references therein. For a study of $H^{\varphi}$-scale on compact manifolds with boundary see~\cite{Kas-19}. For theory of parabolic boundary-value problems on $\RR^n$ in the so-called anisotropic generalized Sobolev spaces, see the recent paper~\cite{LMM-21-paper}, the monograph~\cite{LMM-21}, and references therein.

The present article lies at the intersection of the two research tracks described in the preceding three paragraphs. After defining the (extended) Sobolev scale $H^{\varphi}(\ZZ^n)$, $\varphi\in RO$ (see section~\ref{SS-1-9} for the definition),  we show that the aforementioned interpolation property (and its variants) hold in the setting of (extended) Sobolev spaces on $\ZZ^n$; see theorems~\ref{T:main-1} and~\ref{T:main-2} below. Furthermore, in analogy with a ``quadratic interpolation" result of~\cite{MM-15} for bounded domains $\Omega\subset\RR^n$ with Lipschitz boundary, we show that (see theorem~\ref{T:main-3} below) the class $H^{\varphi}(\ZZ^n)$ is closed under interpolation with a function parameter. In theorem~\ref{T:main-4} we establish additional properties of the scale $H^{\varphi}(\ZZ^n)$, including the density of the Schwartz space $S(\ZZ^n)$ in $H^{\varphi}(\ZZ^n)$ and an embedding result $H^{\varphi}(\ZZ^n)\hookrightarrow \ell^{\infty}(\ZZ^n)$, where $\ell^{\infty}(\ZZ^n)$ stands for bounded functions on $\ZZ^n$.

A property from~\cite{BRK-20} says that a PDO $A$ of order $m\in\RR$ on $\ZZ^n$ (see section~\ref{SS:s-2-5} below for the definition of such PDOs) extends to a bounded linear operator $A\colon H^{(s)}(\ZZ^n)\to H^{(s-m)}(\ZZ^n)$. If, in addition, $A$ is an elliptic operator of order $m$, then $A\colon H^{(s)}(\ZZ^n)\to H^{(s-m)}(\ZZ^n)$ is a Fredholm operator whose index does not depend on $s$ (see theorem 4.2 in~\cite{DK-20} for the case $s=m=0$ and proposition~\ref{P:prop-s-s-m} below for general $s$ and $m$).  With this in mind, we use the interpolation result of theorem~\ref{T:main-2} to prove the corresponding mapping property and the Fredholmness property for PDOs in the scale $H^{\varphi}(\ZZ^n)$, $\varphi\in RO$; see theorems~\ref{T:main-5} and~\ref{T:main-5-f} below.

As in the corresponding definition in~\cite{MM-21} for closed manifolds, working in the setting of $\ZZ^n$,  for $\varphi\in RO$ and for a first-order PDO $A$ of elliptic type (see definition~\ref{D-2} below) satisfying $(Au,u)\geq \|u\|^2$ for all $u\in S(\ZZ^n)$, we define the so-called extended $A$-scale $H_{A}^{\varphi}(\ZZ^n)$. Under these hypotheses, in theorem~\ref{T:main-7} we show that, up to norm equivalence, we have $H_{A}^{\varphi}(\ZZ^n)=H^{\varphi}(\ZZ^n)$.

Lastly, we remark that the analogues of interpolation results of our article hold for the Sobolev scale $H^{(s)}(\hbar\ZZ^n)$ on the lattices $\hbar\ZZ^n$, $\hbar\in (0,1]$, and the corresponding mapping and Fredholmness properties hold for PDOs with $S^{m}_{\rho,\delta}(\hbar\ZZ^n\times\mathbb{T}^n)$-type symbols from~\cite{BCR-23} (here $0\leq\delta<\rho\leq 1$).  Furthermore, the results of our article also carry over to the weighted Sobolev spaces $H^{(s)}_{\Lambda}(\ZZ^n)$ and operators with $M^{m}_{\rho,\Lambda}(\ZZ^n\times\mathbb{T}^n)$-type symbols from~\cite{km-21}. To keep our presentation simpler, we chose to work in the setting of $S^{m}(\ZZ^n\times\mathbb{T}^n)$-type symbols from~\cite{BRK-20, DK-20}.

The article consists of eleven sections and an appendix. In section~\ref{S:res} we summarize the basic notations, define the usual Sobolev scale on $\ZZ^n$, $RO$-varying functions, the extended Sobolev scale on $\ZZ^n$, and PDOs on $\ZZ^n$, and we recall basic concepts concerning the interpolation with a function parameter. Additionally, in section~\ref{S:res} (more specifically subsections~\ref{SS-1-12}--\ref{SS-1-15}, \ref{SS-1-19}, \ref{SS-1-19-f}, and \ref{SS-1-22}) we state the main results (seven theorems and one corollary) of the article. For reader's convenience, in section~\ref{S:S-2} we recalled the statements of a few auxiliary results on interpolation with a function parameter. Sections~\ref{S:S-3}--\ref{S:S-9} contain the proofs of the main results. In the appendix we discuss the anti-duality of the spaces $H^{(s)}(\ZZ^n)$ and $H^{(-s)}(\ZZ^n)$.

\section{Notations and Results}\label{S:res}
\subsection{Basic Notations}\label{SS:s-2-1}
In this paper, the notations $\ZZ$, $\NN$, and $\NN_0$ indicate the sets of integers, positive integers, and non-negative integers respectively. For $n\in\NN$, we denote the $n$-dimensional integer lattice by $\ZZ^n$. For an $n$-dimensional multiindex $\alpha=(\alpha_1, \alpha_2, \dots, \alpha_n)$ with $\alpha_j\in \NN_0$,  we define $|\alpha|:=\alpha_1+ \alpha_2+ \dots+ \alpha_n$, and $\alpha!:=\alpha_1\alpha_2\dots\alpha_n$. For $k=(k_1,k_2,\dots,k_n)\in \ZZ^n$ and $\alpha\in\NN_0^{n}$, we define
\begin{equation*}
k^{\alpha}:=k_1^{\alpha_1}k_2^{\alpha_2}\dots k_n^{\alpha_n}
\end{equation*}
and
\begin{equation*}
|k|:=\sqrt{k_1^2+k_2^2+\dots+k_n^2}.
\end{equation*}

\subsection{Basic Operators}\label{SS:diff-op-intro}
Let $\{e_j\}_{j=1}^{n}$ be a collection of elements such that
\begin{equation*}
e_j:=(0,0,\dots, 1,0, \dots, 0),
\end{equation*}
with $1$ occupying the $j$-th slot and $0$ occupying the remaining slots.

For a function $u(k_1,k_2,\dots,k_n)$ of the input variable $k=(k_1,k_2,\dots, k_n)\in \ZZ^n$, the first partial difference operator $\del_{k_j}$ is defined as
\begin{equation*}
\del_{k_j}u(k):=u(k+e_j)-u(k),
\end{equation*}
where $k+e_j$ is the usual addition of the $n$-tuplets $k$ and $e_j$.
For a multiindex $\alpha\in \NN_0^{n}$, we set
\begin{equation*}
\del^{\alpha}_k:=\del^{\alpha_1}_{k_1}\del^{\alpha_2}_{k_2}\dots \del^{\alpha_n}_{k_n}.
\end{equation*}
We now recall basic differential operators on the $n$-dimensional torus $\mathbb{T}^n:=\RR^n/\ZZ^n$. For $x\in \mathbb{T}^n$ and $\alpha\in \NN^n$, we define
\begin{equation*}
D_{x_j}:=\frac{1}{2\pi i}\frac{\p}{\p{x_j}},\qquad D^{\alpha}_x:=D^{\alpha_1}_{x_1}D^{\alpha_2}_{x_2}\dots D^{\alpha_n}_{x_n},
\end{equation*}
where $i$ is the imaginary unit. Additionally, for $l\in \NN_0$ we define,
\begin{equation*}
D^{(l)}_{x_j}:=\prod_{r=0}^{l-1}\left(\frac{1}{2\pi i}\frac{\p}{\p{x_j}}-r\right),\qquad D^{(0)}_{x_j}:=1,
\end{equation*}
where ``1" refers to the identity operator. For $\alpha\in \NN_0^{n}$, we define
\begin{equation*}
D^{(\alpha)}_{x}:=D^{(\alpha_1)}_{x_1}D^{(\alpha_2)}_{x_2}\dots D^{(\alpha_n)}_{x_n}.
\end{equation*}

\subsection{Schwartz Space}\label{SS:s-2-3} As specified in~\cite{BRK-20} and~\cite{DK-20}, the Schwartz space $\mathcal{S}(\ZZ^n)$ consists of the functions $u\colon\ZZ^n\to\CC$ such that for all $\alpha,\beta\in\NN_0^n$ we have
\begin{equation*}
\displaystyle\sup_{k\in\ZZ^n}|k^{\alpha}(\del_{k}^{\beta}u)(k)|<\infty.
\end{equation*}
The symbol $S'(\ZZ^n)$ indicates the space of tempered distributions, that is, continuous linear functionals on $\mathcal{S}(\ZZ^n)$.

\subsection{Discrete $L^p$-space}\label{SS:s-2-4-l2}
For $1\leq p<\infty$ we define $\ell^p(\ZZ^n)$ as the space of functions $u\colon \ZZ\to \CC$ such that $\|u\|_{p}<\infty$, where
\begin{equation*}
\|u\|_{p}^{p}:=\sum_{k\in\ZZ^n}|u(k)|^p.
\end{equation*}
In particular for $p=2$ we get a Hilbert space $\ell^2(\ZZ^n)$ with the inner product
\begin{equation}\label{E:inner-l-2}
  (u,v):=\sum_{k\in\ZZ^n}u(k)\overline{v(k)}.
\end{equation}
To simplify the notation we will denote the corresponding norm in $\ell^2(\ZZ^n)$ by $\|\cdot\|$.

We define $\ell^{\infty}(\ZZ^n)$ as the space of functions $u\colon \ZZ\to \CC$ such that $\|u\|_{\infty}<\infty$, where
\begin{equation*}
\|u\|_{\infty}:=\sup_{k\in\ZZ^n}|u(k)|.
\end{equation*}

\subsection{Sobolev Scale on $\ZZ^n$}\label{SS:s-2-6} To formulate our results we will need discrete Sobolev spaces as described in section 2 of~\cite{DK-20}.

For $s\in\RR$, we define
\begin{equation}\label{sob-mult}
\langle k\rangle :=(1+|k|^2)^{\frac{1}{2}}, \qquad k\in\ZZ^n.
\end{equation}

Next, for $s\in\RR$ we define
\begin{equation}\label{D:sob-sp}
  H^{(s)}(\ZZ^n):=\{u\in \mathcal{S}'(\ZZ^n)\colon \langle k\rangle^{s} u\in \ell^2(\ZZ^n)\}
\end{equation}
with the norm $\|u\|_{H^{(s)}}:=\|\langle k\rangle^{s}u\|$, where $\|\cdot\|$ is the norm corresponding to the inner product~(\ref{E:inner-l-2}) in $\ell^2(\ZZ^n)$.

\begin{remark}\label{R:weighted} We can view the space $H^{(s)}(\ZZ^n)$ as a weighted space $\ell_{\langle k\rangle^{2s}}^2(\ZZ^n)$, that is, the $\ell^2$-space with weight $\langle k\rangle^{2s}$.
\end{remark}

\begin{remark}\label{R:dense-sobolev}
An important property, established in lemma 3.16 of~\cite{DK-20}, is the density of the space $\mathcal{S}(\ZZ^n)$ in $H^{(s)}(\ZZ^n)$ for all $s\in\RR$.
\end{remark}

The definition of ``extended Sobolev scale," as specified in section 2.4 of~\cite{MM-14}, relies on the so-called $RO$-varying functions.

\subsection{$RO$-varying Functions}\label{SS-1-8} Throughout this section we follow the terminology of section 2.4.1 in~\cite{MM-14}.
We say that a function $\varphi\colon [1,\infty)\to (0,\infty)$ is \emph{$RO$-varying at infinity} if
\begin{enumerate}
  \item [(i)] $\varphi$ is Borel measurable
  \item [(ii)] there exist numbers $a>1$ and $c\geq 1$ (depending on $\varphi$) such that
  \begin{equation}\label{E:RO-1}
    c^{-1}\leq\frac{\varphi(\lambda t)}{\varphi(t)}\leq c, \qquad\textrm{for all }t\geq 1,\,\,\lambda\in [1,a].
  \end{equation}
\end{enumerate}
In the sequel, the inclusion $\varphi\in RO$ means that a function $\varphi\colon [1,\infty)\to (0,\infty)$ is $RO$-varying at infinity.

By proposition 1 of~\cite{MM-13}, if $\varphi\in RO$, then $\varphi$ is bounded and separated from zero on every interval of the form $[1,b]$ with $b>1$. Furthermore, according to the same proposition, the condition~(\ref{E:RO-1}) has the following equivalent formulation: there exist numbers $s_0\leq s_1$ and $c\geq 1$  such that
\begin{equation}\label{E:RO-2}
   t^{-s_0}\varphi(t)\leq c\tau^{-s_0}\varphi(\tau),\qquad \tau^{-s_1}\varphi(\tau)\leq ct^{-s_1}\varphi(t) \qquad\textrm{for all }1\leq t\leq \tau.
  \end{equation}

To conclude this section, we review the concept of lower/upper Matuszewska indices of $\varphi\in RO$. Let $\varphi\in RO$. Setting $\lambda:=\frac{\tau}{t}$, we can rewrite~(\ref{E:RO-2}) as
\begin{equation}\label{E:RO-3}
    c^{-1}\lambda^{s_0}\leq\frac{\varphi(\lambda t)}{\varphi(t)}\leq c\lambda^{s_1}, \qquad\textrm{for all }t\geq 1,\,\,\lambda\geq 1.
  \end{equation}
We define the \emph{lower Matuszewska index} $\sigma_0(\varphi)$ as the supremum of all $s_0\in\RR$ such that the leftmost inequality in~(\ref{E:RO-3}) is satisfied. Likewise, we define the \emph{upper Matuszewska index} $\sigma_1(\varphi)$ as the infimum of all $s_1\in\RR$ such that the rightmost inequality in~(\ref{E:RO-3}) is satisfied. Note that $-\infty<\sigma_0(\varphi)\leq \sigma_1(\varphi)<\infty$.

\subsection{Extended Sobolev Scale on $\ZZ^n$}\label{SS-1-9} For $\varphi\in RO$ we define $H^{\varphi}(\ZZ^n)$ as
\begin{equation}\label{E:sob-phi-1}
H^{\varphi}(\ZZ^n):=\{u\in S'(\ZZ^n)\colon \varphi(\langle k \rangle)u\in \ell^2(\ZZ^n)\}.
\end{equation}
As in the case of the spaces $H^{(s)}(\ZZ^n)$, it turns out that $H^{\varphi}(\ZZ^n)$ is a Hilbert space with the inner product
\begin{equation}\label{E:sob-phi-2}
(u,v)_{H^{\varphi}(\ZZ^n)}:=\sum_{k\in\ZZ} [\varphi(\langle k \rangle)]^2 u(k)\overline{v(k)}
\end{equation}
and the norm corresponding to~(\ref{E:sob-phi-2}) will be denoted by $\|\cdot\|_{H^{\varphi}(\ZZ^n)}$.

Note that if $\varphi(t)=t^{s}$, $s\in\RR$, the space $H^{\varphi}(\ZZ^n)$ leads to the Sobolev space $H^{(s)}(\ZZ^n)$. In this article, by \emph{extended Sobolev scale} on $\ZZ^n$ we mean the class of spaces $\{H^{\varphi}(\ZZ^n)\colon \varphi\in \RR\}$.

In the next section we review some terminology from sections 1.1.1 and 1.1.2 of~\cite{MM-14} concerning interpolation of a pair of Hilbert spaces with a function parameter.

\subsection{Interpolation Between Hilbert Spaces}\label{SS-2-1} By an \emph{admissible pair} of separable complex Hilbert spaces we mean an ordered pair $[\mathscr{H}_{0},\mathscr{H}_{1}]$ such that $\mathscr{H}_{1}\hookrightarrow\mathscr{H}_{0}$, with the embedding being continuous and dense.  As indicated in section 1.2.1 of~\cite{Lions-72}, every admissible pair $[\mathscr{H}_{0},\mathscr{H}_{1}]$ is equipped with a so-called \emph{generating operator} $J$, such that
\begin{enumerate}
  \item [(i)] $J$ is a positive self-adjoint operator in $\mathscr{H}_{0}$ with $\dom(J)=\mathscr{H}_{1}$;
   \item [(ii)] $\|Ju\|_{\mathscr{H}_{0}}=\|u\|_{\mathscr{H}_{1}}$, for all $u\in \dom(J)=\mathscr{H}_{1}$.
\end{enumerate}

According to section 1.1.1 in~\cite{MM-14}, the operator $J$ is uniquely determined by the admissible pair $[\mathscr{H}_{0},\mathscr{H}_{1}]$.

Let $\mathcal{B}$ be the set of all Borel measurable functions $\psi\colon (0,\infty)\to (0,\infty)$ satisfying the following two properties: $\psi$ is bounded that on every interval $[a,b]$ with $0 <a<b<\infty$, and $\frac{1}{\psi}$ is bounded on every interval $(c,\infty)$ with $c > 0$. For an admissible pair $\mathscr{H}:=[\mathscr{H}_{0},\mathscr{H}_{1}]$ with generating operator $J$ and for a function $\psi\in \mathcal{B}$, spectral calculus gives rise to a (positive self-adjoint) operator $\psi(J)$ in $\mathscr{H}_{0}$. We define a (separable, Hilbert) space $[\mathscr{H}_{0},\mathscr{H}_{1}]_{\psi}$ (or, in abbreviated form, $\mathscr{H}_{\psi}$) as follows: $[\mathscr{H}_{0},\mathscr{H}_{1}]_{\psi}:=\dom(\psi(J))$ with the inner product
\begin{equation}\label{E:H-psi}
(u,v)_{\mathscr{H}_{\psi}}:=(\psi(J)u,\psi(J)v)_{\mathscr{H}_{0}},
\end{equation}
and the corresponding norm $\|u\|_{\mathscr{H}_{\psi}}:=\|\psi(J)u\|_{\mathscr{H}_{0}}$,
where $(\cdot,\cdot)_{\mathscr{H}_{0}}$ and $\|\cdot\|_{\mathscr{H}_{0}}$ are the inner product and the norm in ${\mathscr{H}_{0}}$.

Having defined the space $[\mathscr{H}_{0},\mathscr{H}_{1}]_{\psi}$, we proceed to describe the notion of interpolation parameter $\psi\in\mathcal{B}$. We say that a function $\psi\in\mathcal{B}$ is an \emph{interpolation parameter} if the following condition is satisfied for all admissible pairs $\mathscr{H}:=[\mathscr{H}_{0},\mathscr{H}_{1}]$ and $\mathscr{K }:=[\mathscr{K}_{0},\mathscr{K}_{1}]$ and for all linear operators $T$ with $\mathscr{H}_{0}\subseteq \dom(T)$: if the restrictions $T|_{\mathscr{H}_{0}}$ and $T|_{\mathscr{H}_{1}}$ act as bounded linear operators $T\colon \mathscr{H}_{0}\to \mathscr{K}_{0}$ and $T\colon \mathscr{H}_{1}\to \mathscr{K}_{1}$, then the restriction
$T|_{\mathscr{H}_{\psi}}$ acts as a bounded linear operator $T\colon \mathscr{H}_{\psi}\to \mathscr{K}_{\psi}$.

In this case, we say that the space $\mathscr{H}_{\psi}$ is obtained by interpolation of the pair $\mathscr{H}=[\mathscr{H}_{0},\mathscr{H}_{1}]$ with a function parameter $\psi\in\mathcal{B}$. Moreover, we have the following continuous dense embeddings: $\mathscr{H}_{1}\hookrightarrow\mathscr{H}_{\psi}\hookrightarrow\mathscr{H}_{0}$.

\subsection{Interpolation Space}\label{SS-1-11} Let $[\mathscr{H}_{0},\mathscr{H}_{1}]$ be an ordered pair  of separable complex Hilbert spaces such that $\mathscr{H}_{1}\hookrightarrow\mathscr{H}_{0}$, where the arrow stands for continuous embedding. We say that a Hilbert space $\mathscr{S}$ is an \emph{interpolation space with respect to a pair} $[\mathscr{H}_{0},\mathscr{H}_{1}]$ if the following conditions are satisfied:
\begin{enumerate}
  \item [(i)] we have continuous embeddings $\mathscr{H}_{1}\hookrightarrow\mathscr{S}\hookrightarrow\mathscr{H}_{0}$;
  \item [(ii)] any linear operator $T$ in $\mathscr{H}_{0}$ which acts as a bounded linear operator $T\colon \mathscr{H}_{0}\to \mathscr{H}_{0}$ and $T\colon \mathscr{H}_{1}\to \mathscr{H}_{1}$, has the property that $T\colon \mathscr{S}\to\mathscr{S}$ is also a bounded linear operator.
\end{enumerate}
\begin{remark}\label{R-rem-int} Part (ii) of the above definition leads to the following property (see theorem 1.8 in~\cite{MM-14} or theorem 2.4.2 in~\cite{BL-book}):
\begin{equation*}
  \|T\|_{\mathscr{S}\to \mathscr{S}}\leq C \max \{\|T\|_{\mathscr{H}_0\to \mathscr{H}_0},\|T\|_{\mathscr{H}_1\to \mathscr{H}_1}\},
\end{equation*}
where $C>0$ is a constant independent of $T$.
\end{remark}

\subsection{First Interpolation Result}\label{SS-1-12} Our first result concerns interpolation of a pair of (usual) Sobolev spaces $[H^{(s_0)}(\ZZ^n), H^{(s_1)}(\ZZ^n)]$, where $s_0<s_1$ are real numbers. This result is an analogue of theorem 1 of~\cite{MM-13} concerning the pair $[H^{(s_0)}(\RR^n), H^{(s_1)}(\RR^n)]$, $s_0<s_1$.

Before stating the result, we recall (see proposition 3.4 in~\cite{DK-20}) that there is a continuous (and dense) embedding $H^{(s_1)}(\ZZ^n)\hookrightarrow H^{(s_0)}(\ZZ^n)$, $s_0<s_1$.


\begin{theorem}\label{T:main-1} The following are equivalent:
\begin{enumerate}
  \item [(i)] A Hilbert space $\mathscr{S}$ is an interpolation space with respect to a pair\\
  $[H^{(s_0)}(\ZZ^n), H^{(s_1)}(\ZZ^n)]$, where $s_0<s_1$ are some real numbers.
  \item [(ii)] Up to norm equivalence, we have  $\mathscr{S}=H^{\varphi}(\ZZ^n)$, for some function $\varphi\in RO$ satisfying the condition~(\ref{E:RO-3}).
\end{enumerate}
\end{theorem}

For future reference, we recall the definition of an interpolation space with respect to a scale of Hilbert spaces. Let $\{\mathscr{H}_{s}\colon s\in\RR\}$ be a scale of Hilbert spaces such that there is continuous embedding $\mathscr{H}_{s_1}\hookrightarrow\mathscr{H}_{s_0}$ for all $s_0<s_1$. We say that a Hilbert space $\mathscr{S}$ is an \emph{interpolation space with respect to the scale} $\{\mathscr{H}_{s}\colon s\in\RR\}$ if there exist numbers $s_0<s_1$ such that $\mathscr{S}$ is an interpolation space with respect to the pair $[\mathscr{H}_{s_0},\mathscr{H}_{s_1}]$.

\begin{corollary} \label{C:cor-1} The following are equivalent:
\begin{enumerate}
  \item [(i)] A Hilbert space $\mathscr{S}$ is an interpolation space with respect to the scale\\
  $\{H^{(s)}(\ZZ^n)\colon s\in\RR\}$.
  \item [(ii)] Up to norm equivalence, we have  $\mathscr{S}=H^{\varphi}(\ZZ^n)$, for some function $\varphi\in RO$.
\end{enumerate}
\end{corollary}

\subsection{Second Interpolation Result}\label{SS-1-13} The result below captures the implication (ii)$\implies$(i) of theorem~\ref{T:main-1} in a more explicit form. For an analogous result in the context of a bounded domain $\Omega \subset\RR^n$ with Lipschitz boundary, see theorem 5.1 of~\cite{MM-15}.
\begin{theorem}\label{T:main-2}  Let $\varphi\in RO$, $s_0<\sigma_0(\varphi)$, and $s_1>\sigma_1(\varphi)$. Define $\psi$ as follows:
\begin{equation}\label{E:phi-psi-inverse}
\psi(t):=\left\{\begin{array}{cc}
                  \tau^{-s_0/(s_1-s_0)}\varphi(\tau^{1/(s_1-s_0)}), & \tau\geq 1, \\
                  \varphi(1), & 0<\tau<1,
                \end{array}\right.
\end{equation}
Then, $\psi\in\mathcal{B}$ and $\psi$ is an interpolation parameter. Furthermore, with equality of norms, we have
\begin{equation*}
[H^{(s_0)}(\ZZ^n), H^{(s_1)}(\ZZ^n)]_{\psi}=H^{\varphi}(\ZZ^n).
\end{equation*}
\end{theorem}
\subsection{Quadratic Interpolation}\label{SS-1-14} Our next result shows that $\{H^{\varphi}(\ZZ^n)\colon \varphi\in RO\}$ is closed with respect to interpolation with a function parameter. This result is an analogue of theorem 5.2 of~\cite{MM-15} for the class $\{H^{\varphi}(\Omega)\colon \varphi\in RO\}$ where $\Omega \subset\RR^n$ is a bounded domain with Lipschitz boundary.

\begin{theorem}\label{T:main-3}  Assume that $\varphi_0,\,\varphi_1\in RO$ and $\frac{\varphi_0}{\varphi_1}$ is bounded in a neighborhood of $\infty$. Let $\psi\in\mathcal{B}$ be an interpolation parameter. Then, the following hold:
\begin{enumerate}
  \item [(i)] $[H^{\varphi_0}(\ZZ^n), H^{\varphi_1}(\ZZ^n)]$ is an admissible pair;
  \item [(ii)] up to the norm equivalence we have
  \begin{equation}\label{E:t-3-main}
[H^{\varphi_0}(\ZZ^n), H^{\varphi_1}(\ZZ^n)]_{\psi}=H^{\varphi}(\ZZ^n),
\end{equation}
\end{enumerate}
where
\begin{equation}\label{E:quad-int}
\varphi(t):=\varphi_0(t)\psi\left(\frac{\varphi_1(t)}{\varphi_0(t)}\right).
\end{equation}
\end{theorem}

\subsection{Further Properties of $H^{\varphi}(\ZZ^n)$}\label{SS-1-15} The next theorem summarizes additional properties of the class $\{H^{\varphi}(\ZZ^n)\colon \varphi\in RO\}$. In the setting of $\RR^n$, analogous properties were demonstrated in proposition 2 of~\cite{MM-13}.

\begin{theorem}\label{T:main-4} We have the following properties:
\begin{enumerate}

\item [(i)] Assume that $\varphi\in RO$. Then $\mathcal{S}(\ZZ^n)$ is dense in $H^{\varphi}(\ZZ^n)$.

\item [(ii)] Assume that $\varphi_0,\,\varphi_1\in RO$ and $\frac{\varphi_0}{\varphi_1}$ is bounded in a neighborhood of $\infty$. Then we have a continuous embedding $H^{\varphi_1}(\ZZ^n)\hookrightarrow H^{\varphi_0}(\ZZ^n)$.

\item [(iii)] Assume that $\varphi\in RO$. Then, the sesquilinear form~(\ref{E:inner-l-2}) extends to a sesquilinear duality (separately continuous sesquilinear form)
\begin{equation}\label{E:inner-product-h-phi}
(\cdot,\cdot)\colon H^{\varphi}(\ZZ^n)\times H^{\frac{1}{\varphi}}(\ZZ^n)\to\mathbb{C}.
\end{equation}
The spaces $H^{\varphi}(\ZZ^n)$ and $H^{\frac{1}{\varphi}}(\ZZ^n)$ are mutually dual relative to~(\ref{E:inner-product-h-phi}).
\item [(iv)] Assume that $\varphi\in RO$ satisfies
\begin{equation}\label{E:sum-phi-hyp}
\sum_{k\in\ZZ^n}\frac{1}{[\varphi(\langle k\rangle)]^2}<\infty,
\end{equation}
where $\langle k\rangle$ is as in~(\ref{sob-mult}).

Then, we have a continuous embedding $H^{\varphi}(\ZZ^n)\hookrightarrow \ell^{\infty}(\ZZ^n)$.
\end{enumerate}
\end{theorem}

In our next result we establish a mapping property of pseudo-differential operators (PDO) acting on  $H^{\varphi}(\ZZ^n)$. Before stating this result, we recall some elements of  PDO calculus on $\ZZ^n$, as developed by the authors of~\cite{BRK-20}.

\subsection{Symbol Classes}\label{SS:s-2-2} We begin with the definition of the symbol class $\sml$, as introduced in~\cite{BRK-20}.
\begin{definition}\label{D-1} For $m\in\RR$, the notation $S^{m}(\ZZ^n\times \TT^n)$ indicates the set of functions $a\colon \ZZ^n\times \TT^n\to\mathbb{C}$ satisfying the following properties:
\begin{enumerate}
\item [(i)] for all $k\in \ZZ^n$, we have $a(k,\cdot)\in C^{\infty}(\TT^n)$;
\item [(ii)] for all $\alpha,\,\beta\in \NN_0^{n}$, there exists a constant $C_{\alpha,\beta}>0$ such that
\begin{equation*}
 |D_{x}^{(\beta)}\Delta^{\alpha}_{k}a(k,x)|\leq C_{\alpha,\beta}(1+|k|)^{m-|\alpha|},
\end{equation*}
for all $(k,x)\in \ZZ^n\times \TT^n$.
\end{enumerate}
\end{definition}

We new recall the definition of an elliptic symbol from~\cite{BRK-20}.

\begin{definition}\label{D-2}  For $m\in\RR$, the elliptic symbol class $ES^{m}(\ZZ^n\times \TT^n)$ refers to the set of functions $a\in S^{m}(\ZZ^n\times \TT^n)$  satisfying the following property: there exist constants $C>0$ and $R>0$ such that
\begin{equation*}
|a(k,x)|\geq C(1+|k|)^m,
\end{equation*}
for all $x\in\TT^n$ and all $k\in \ZZ^n$ such that $|k|>R$.
\end{definition}

An important ingredient in the definition of a PDO on $\ZZ^n$ is the discrete Fourier transform, which we describe next.

\subsection{Discrete Fourier Transform}\label{SS:s-2-4}
For $u\in \ell^1(\ZZ^n)$, its discrete Fourier transform $\widehat{u}(x)$ is a function of $x\in\TT^n$ defined as
\begin{equation*}
\widehat{u}(x):=\sum_{k\in\ZZ^n}e^{-2\pi i k\cdot x}u(k),
\end{equation*}
where $k\cdot x:=k_1x_1+k_2x_2+\dots +k_nx_n$.
It turns out that the discrete Fourier transform can be extended to $\ell^2(\ZZ^n)$, and by normalizing the Haar measure on $\ZZ^n$ and $\TT^n$, the Plancherel formula takes the following form:
\begin{equation*}
\sum_{k\in\ZZ^n}|u(k)|^2=\int_{\TT^n}|\widehat{u}(x)|^2\,dx.
\end{equation*}
The corresponding inversion formula looks as follows:
\begin{equation}\label{E:inv-ft}
u(k)=\int_{\TT^n}e^{2\pi i k\cdot x}\widehat{u}(x)\,dx,\qquad k\in\ZZ^n.
\end{equation}

\subsection{Pseudo-Differential Operator}\label{SS:s-2-5} The pseudo-differential operator $T_{a}$ corresponding to $a\in S^m(\ZZ^n\times \TT^n)$, also denoted as $\textrm{Op}[a]$, is defined as
\begin{equation}\label{E:op-a}
  (T_{a}u)(k):= \int_{\TT^n} e^{2\pi i k\cdot x}a(k,x)\widehat{u}(x)\,dx,\quad u\in \mathcal{S}(\ZZ^n).
\end{equation}

In proposition 3.15 of~\cite{DK-20}, the authors showed that the operator $T_{a}$ maps $\mathcal{S}(\ZZ^n)$ into $\mathcal{S}(\ZZ^n)$.
For a linear operator $P\colon \mathcal{S}(\ZZ^n)\to  \mathcal{S}(\ZZ^n)$, its  \emph{formal adjoint} $P^{\dagger}$ is defined using the following relation:
\begin{equation}\label{E:dual-1}
  (Pu,v)=(u, P^{\dagger} v),
\end{equation}
for all $u,\,v\in \mathcal{S}(\ZZ^n)$, where $(\cdot,\cdot)$ is as in~(\ref{E:inner-l-2}).

Finally, we note that the (linear) operator $T_{a}\colon \mathcal{S}(\ZZ^n)\to \mathcal{S}(\ZZ^n)$ extends to a (linear) operator
\begin{equation*}
T_{a}\colon S'(\ZZ^n)\to S'(\ZZ^n)
\end{equation*}
defined as follows:
\begin{equation*}
(T_{a}F)(\overline{u}):=F\left(\overline{T_{a}^{\dagger}u}\right), \qquad F\in \mathcal{S}'(\ZZ^n),\,\,u\in \mathcal{S}(\ZZ^n).
\end{equation*}
\noindent (Here, $T_{a}^{\dagger}$ is the formal adjoint of $T_{a}$ and $\overline{z}$ is the conjugate of $z\in\CC$.)

\subsection{Mapping Property of PDO on $\ZZ^n$}\label{SS-1-19} Our mapping property is an analogue of the property stated in section 4 of~\cite{MM-13} for the $\RR^n$-setting.

\begin{theorem}\label{T:main-5} Assume that $a\in S^{m}(\ZZ^n\times \TT^n)$, where $m\in\RR$. Let $\varphi\in RO$. Then $T_{a}$ extends to a bounded linear operator
\begin{equation}\nonumber
T_{a}\colon H^{\varphi}(\ZZ^n)\to H^{t^{-m}\varphi}(\ZZ^n),
\end{equation}
where $t^{-m}\varphi$ denotes the product of the functions $t^{-m}$ and $\varphi$.
\end{theorem}

\subsection{Fredholmness of PDO on $\ZZ^n$}\label{SS-1-19-f} Let $\mathscr{B}_{1}$ and $\mathscr{B}_{2}$ be Banach spaces and let $T\colon \mathscr{B}_{1}\to \mathscr{B}_{2}$ be a bounded linear operator. As in defintion 8.1 of~\cite{sh-book}, $T$ is said to be a \emph{Fredhlom operator} if the following conditions are satisfied: (i) $\dim (\textrm{Ker }T)<\infty$, (ii) $\textrm{Ran }T$ is closed, and (iii)  $\dim (\textrm{Coker }T)<\infty$, where  $\textrm{Coker }T:=\mathscr{B_2}/(\textrm{Ran } T)$.

\begin{remark} Actually, for a bounded linear operator $T\colon \mathscr{B}_{1}\to \mathscr{B}_{2}$, the condition $\dim (\textrm{Coker }T)<\infty$ implies that $\textrm{Ran }T$ is closed; see lemma 8.1 in~\cite{sh-book}.
\end{remark}

For a Fredholm operator $T\colon \mathscr{B}_{1}\to \mathscr{B}_{2}$, we define the \emph{index} $\kappa(T)$ as
\begin{equation*}
\kappa(T):=\dim (\textrm{Ker } T)-\dim (\textrm{Coker }T).
\end{equation*}

Let $a\in S^{m}(\ZZ^n\times \TT^n)$ with $m\in\RR$, and let $T_{a}$ and $T_{a}^{\dagger}$ be as in section~\ref{SS:s-2-5}. In the formulation of the theorem we will use the following sets:

\begin{equation}\label{E:ker-a}
\mathscr{K}:=\{u\in H^{\infty}(\ZZ^n)\colon T_{a}u=0\},\qquad \mathscr{K^{\dagger}}:=\{u\in H^{\infty}(\ZZ^n)\colon (T_{a})^{\dagger}u=0\},
\end{equation}
where $H^{\infty}(\ZZ^n):=\cap_{r\in\RR}H^{(r)}(\ZZ^n)$.

The following theorem is a discrete analogue of theorem 2.28 in~\cite{MM-14} (or theorem 3 in~\cite{MM-09}), which was situated in the setting of compact manifolds:

\begin{theorem}\label{T:main-5-f} Assume that $a\in ES^{m}(\ZZ^n\times \TT^n)$, where $m\in\RR$ and $ES^{m}(\ZZ^n\times \TT^n)$ is as in definition~\ref{D-2}. Let $\varphi\in RO$ and let $t^{-m}\varphi$ denote the product of the functions $t^{-m}$ and $\varphi$. Then
\begin{equation}\nonumber
T_{a}\colon H^{\varphi}(\ZZ^n)\to H^{t^{-m}\varphi}(\ZZ^n),
\end{equation}
is a (bounded) Fredholm operator. Furthermore, keeping in mind the notations~(\ref{E:ker-a}), the following properties hold:
\begin{itemize}
  \item [(i)] $\textrm{Ker } T_{a}=\mathscr{K}$;
  \item [(ii)] $\textrm{Ran } T_{a}=\{v\in H^{t^{-m}\varphi}(\ZZ^n)\colon (v,w)=0,\,\,\textrm{for all }w\in \mathscr{K^{\dagger}}\}$, where $(\cdot,\cdot)$ is the sesquilinear duality between $H^{t^{-m}\varphi}(\ZZ^n)$ and $H^{\frac{t^{m}}{\varphi}}(\ZZ^n)$, as described in~(\ref{E:inner-product-h-phi});
  \item [(iii)] the index of $T_{a}$ is $\kappa(T_{a})=\dim (\mathscr{K})-\dim (\mathscr{K^{\dagger}})$. (Hence, $\kappa(T_{a})$ is independent of $\varphi$.)
\end{itemize}
 \end{theorem}

In our last result we use the concept of ``extended Hilbert scale" from~\cite{MM-21}.

\subsection{Extended Hilbert Scale}\label{SS-1-21} In this section we follows the terminology from section 2 of~\cite{MM-21}.
Let $\mathscr{H}$ be a separable complex Hilbert space with inner product $(\cdot,\cdot)_{\mathscr{H}}$ and norm $\|\cdot\|_{\mathscr{H}}$. Let $A$ be a self-adjoint operator in $\mathscr{H}$ such that $(Au,u)_{\mathscr{H}}\geq\|u\|^2_{\mathscr{H}}$ for all $u\in\dom (A)$.

Using spectral calculus we define the operator $A^{s}$ for each $s\in\RR$. Note that $\dom(A^s)$ is dense in $\mathscr{H}$; in particular, if $s\leq0$ we have  $\dom (A^s)=\mathscr{H}$. The space $H_{A}^{(s)}$ is defined as the completion of $\dom(A^s)$ with respect to the inner product
\begin{equation*}
(u,v)_{s}:=(A^su,A^sv)_{\mathscr{H}}, \qquad u,v\in \dom(A^s).
\end{equation*}
It turns out that $H_{A}^{(s)}$ is a separable Hilbert space whose inner product and norm will be denoted by $(\cdot,\cdot)_{s}$ and $\|\cdot\|_{s}$. We call $\{H_{A}^{(s)}\colon s\in\RR\}$ \emph{Hilbert scale generated by $A$} or, in shorter form, \emph{$A$-scale}. As mentioned in section 2 of ~\cite{MM-21}, for $s\geq 0$ we have $H_{A}^{(s)}=\dom (A^s)$, while for $s<0$ we have $H_{A}^{(s)}\supset \mathscr{H}$.

According to section 2 of~\cite{MM-21}, for all $s_0<s_1$,  $[H_{A}^{(s_0)},H_{A}^{(s_1)}]$ is an admissible pair (in the sense of section~\ref{SS-2-1} above). Recalling the definition of interpolation space (see section~\ref{SS-1-11} above), the term \emph{extended Hilbert scale generated by $A$} or
\emph{extended $A$-scale} refers to the set of all Hilbert spaces that serve as interpolation spaces with respect to pairs of the form $[H_{A}^{(s_0)},H_{A}^{(s_1)}]$, $s_0<s_1$.

Remembering the hypotheses on  $A$ and using spectral calculus, for a Borel measurable function $\varphi\colon [1,\infty)\to(0,\infty)$,  we can define a (positive self-adjoint) operator $\varphi(A)$ in $\mathscr{H}$. Furthermore, we define $H_{A}^{\varphi}$ as the completion of
$\dom(\varphi(A))$ with respect to the inner product
\begin{equation*}
(u,v)_{\varphi}:=(\varphi(A)u,\varphi(A)v)_{\mathscr{H}}, \qquad u,v\in \dom(\varphi(A)).
\end{equation*}
According to section 2 of~\cite{MM-21}, $H_{A}^{\varphi}$ is a separable Hilbert space, and its inner product and norm will be denoted by $(\cdot,\cdot)_{\varphi}$ and $\|\cdot\|_{\varphi}$. Furthermore (see section 2 of~\cite{MM-21}), we have $H_{A}^{\varphi}=\dom(\varphi(A))$ if and only if $0\notin\textrm{Spec}(\varphi(A))$.

\subsection{Extended Hilbert Scale on $\ZZ^n$}\label{SS-1-22}
We make the following assumptions on our operator $A:=\textrm{Op}[a]$:
\begin{enumerate}
\item [(H1)] $a\in ES^{1}(\ZZ^n\times \TT^n)$, where $ES^{1}(\ZZ^n\times \TT^n)$ is as in definition~\ref{D-2};

\item [(H2)] $(Au,u)\geq \|u\|^2$, for all $u\in \mathcal{S}(\ZZ^n)$, where $(\cdot,\cdot)$ is as in~(\ref{E:inner-l-2}) and $\|\cdot\|$ is the corresponding norm in $\ell^2(\ZZ^n)$.
\end{enumerate}

\begin{remark}\label{R:kor} Under the assumptions (H1)--(H2), $A|_{\mathcal{S}(\ZZ^n)}$ is a formally self-adjoint operator whose symbol $a$ belongs to the class $ES^{1}(\ZZ^n\times \TT^n)$. Therefore (as a consequence of theorem 3.19 in~\cite{DK-20}), the operator $A|_{\mathcal{S}(\ZZ^n)}$ is essentially self-adjoint in $\ell^2(\ZZ^n)$.  Furthermore, by proposition 3.18 from~\cite{DK-20}, the domain of the self-adjoint closure of $A$ is the space $H^{(1)}(\ZZ^n)$. To keep our notations simpler, we denote the self-adjoint closure of $A|_{\mathcal{S}(\ZZ^n)}$ again by $A$. With this clarification, as in section~\ref{SS-1-21} above (now in the setting $\mathscr{H}=\ell^2(\ZZ^n)$), for a Borel function
$\varphi\colon [1,\infty)\to(0,\infty)$ we define the space $H_{A}^{\varphi}$ corresponding to an operator $A$ satisfying (H1)--(H2).
\end{remark}
The theorem below is an analogue of theorem 5.1 in~\cite{MM-21} for the $\RR^n$-setting.
\begin{theorem}\label{T:main-7} Let $\varphi\in RO$. Assume that $A:=\textrm{Op}[a]$ is an operator satisfying (H1)--(H2). Then, up to norm equivalence, we have
\begin{equation}\label{E:thm-7}
H^{\varphi}_{A}(\ZZ^n)=H^{\varphi}(\ZZ^n),
\end{equation}
where $H^{\varphi}_{A}(\ZZ^n)$ is as in remark~\ref{R:kor} and $H^{\varphi}(\ZZ^n)$ is as in section~\ref{SS-1-9}.
\end{theorem}

\begin{remark} The condition $a\in ES^{1}(\ZZ^n\times \TT^n)$ in (H1) can be replaced by $a\in ES^{m}(\ZZ^n\times \TT^n)$, with $m>0$. In this case, the following variant of~(\ref{E:thm-7}) holds: $H^{\varphi}_{A}(\ZZ^n)=H^{\varphi_{m}}(\ZZ^n)$, where $\varphi_{m}(t):=\varphi(t^m)$, $t\geq 1$.
\end{remark}

\section{Auxiliary Results}\label{S:S-2} In this section we collect various results from literature that are used in proofs of main results of this paper.

For future reference, we say that a Borel measurable function $\psi\colon (0,\infty)\to (0,\infty)$ is \emph{pseudoconcave in a neighborhood of} $\infty$ if there exists a concave function $\psi_1 \colon (c,\infty)\to (0,\infty)$, where $c>1$ is a large number, such that $\frac{\psi}{\psi_1}$ and $\frac{\psi_1}{\psi}$ are bounded on $(c,\infty)$.

We begin with a proposition for which we refer to theorem 1.9 in~\cite{MM-14}.
\begin{prop}\label{L-1} A function $\psi\in\mathcal{B}$ is an interpolation parameter if and only if $\psi$ is pseudoconcave in a neighborhood of $\infty$.
\end{prop}

The next proposition, for which we refer to theorem 11.4.1 in~\cite{O-84}, plays an important role in showing the implication (i)$\implies$(ii) of theorem~\ref{T:main-1}.

\begin{prop}\label{L-2} Let $\mathscr{H}=[\mathscr{H}_{0},\mathscr{H}_{1}]$ be an admissible pair of Hilbert spaces. Assume that a space $\mathscr{S}$ is an interpolation space for the pair $\mathscr{H}$. Then, up to the norm equivalence, we have $\mathscr{S}=\mathscr{H}_{\psi}$, for some function $\psi\in\mathcal{B}$ such that $\psi$ is pseudoconcave in a neighborhood of $\infty$.
\end{prop}

For the proof of the following proposition, see theorem 4.2 in~\cite{MM-15}:

\begin{prop}\label{L-3-b} Let $s_0<s_1$ be two real numbers and let $\psi\in\mathcal{B}$.
Define
\begin{equation}\label{E:phi-psi}
\varphi(t):=t^{s_0}\psi(t^{s_1-s_0}),\quad t\geq 1.
\end{equation}
Then, the following are equivalent:
\begin{enumerate}
  \item [(i)] $\psi$ is an interpolation parameter;
  \item [(ii)] $\varphi$ satisfies~(\ref{E:RO-3}) with some constant $c\geq 1$ independent of $t$ and $\lambda$.
\end{enumerate}
\end{prop}

The key instrument for justifying the implication (ii)$\implies$(i) of theorem~\ref{T:main-1} will be the next proposition (see the proof of the sufficiency part of theorem 2.4 in~\cite{MM-15}):

\begin{prop}\label{L-4} Assume that $\varphi\in RO$ satisfies the condition~(\ref{E:RO-3}). Define $\psi$ as follows:
\begin{equation}\label{E:phi-psi-inverse}
\psi(t):=\left\{\begin{array}{cc}
                  \tau^{-s_0/(s_1-s_0)}\varphi(\tau^{1/(s_1-s_0)}), & \tau\geq 1, \\
                  \varphi(1), & 0<\tau<1,
                \end{array}\right.,
\end{equation}
where $s_0<s_1$ are as in~(\ref{E:RO-3}). Then the following properties hold:
\begin{enumerate}
  \item [(i)] $\psi\in \mathcal{B}$ and $\psi$ satisfies~(\ref{E:phi-psi});
  \item [(ii)] $\psi$ is an interpolation parameter.
\end{enumerate}
\end{prop}

\section{Proof of Theorem~\ref{T:main-1}}\label{S:S-3}
We first prove a key proposition:
\begin{prop}\label{L-3} Let $s_0<s_1$ be two real numbers and let $\psi\in\mathcal{B}$ be an interpolation parameter. Let $\varphi$ be as in~(\ref{E:phi-psi}).

Then the following properties hold:
\begin{enumerate}
  \item [(i)] $\varphi\in RO$;
  \item [(ii)] we have (with equality of norms)
\begin{equation}\label{interp-r-n}
[H^{(s_0)}(\ZZ^n), H^{(s_1)}(\ZZ^n)]_{\psi}=H^{\varphi}(\ZZ^n).
\end{equation}
\end{enumerate}
\end{prop}
\begin{proof}
Part (i) was proved in lemma 1 of~\cite{MM-13}. To prove part (ii), we use the definition of the space
$[H^{(s_0)}(\ZZ^n), H^{(s_1)}(\ZZ^n)]_{\psi}$ from section~\ref{SS-2-1}. First, we recall (see proposition 3.4 in~\cite{DK-20} and remark~\ref{R:dense-sobolev}) that there is a continuous (and dense) embedding $H^{(s_1)}(\ZZ^n)\hookrightarrow H^{(s_0)}(\ZZ^n)$. Thus, $[H^{(s_0)}(\ZZ^n), H^{(s_1)}(\ZZ^n)]$ is an admissible pair.

Remembering the definitions of the norms in $H^{(s_0)}(\ZZ^n)$ and $H^{(s_1)}(\ZZ^n)$ and taking into account remark~\ref{R:weighted}, we see that the multiplication operator $Ju:=\langle k\rangle^{s_1-s_0}u$ by the function $\langle k\rangle^{s_1-s_0}$ constitutes a generating operator for the pair $[H^{(s_0)}(\ZZ^n), H^{(s_1)}(\ZZ^n)]$.

As $J$ is a multiplication operator by the function $\langle k\rangle^{s_1-s_0}$, it follows that $\psi(J)$ is a multiplication operator by the function $\psi(\langle k\rangle^{s_1-s_0})$.
Denoting
\begin{equation*}
\mathscr{H}:=[H^{(s_0)}(\ZZ^n), H^{(s_1)}(\ZZ^n)]
\end{equation*}
and referring to the definitions~(\ref{E:H-psi}), (\ref{D:sob-sp}) and~(\ref{E:sob-phi-2}), for all $u\in S(\ZZ^n)$ we have
\begin{align}
&\|u\|^2_{\mathscr{H}_{\psi}}=\|\psi(J)u\|^2_{H^{(s_0)}}=\|\psi(\langle k\rangle^{s_1-s_0})u\|^2_{H^{(s_0)}}\nonumber\\
&=\|\langle k\rangle^{s_0}\psi(\langle k\rangle^{s_1-s_0})u\|^2=\|\varphi(\langle k\rangle)u\|^2=\|u\|^2_{H^{\varphi}},
\end{align}
where the fourth equality follows from~(\ref{E:phi-psi}).
Thus, we obtain~(\ref{interp-r-n}) with equality of norms. (By remark~\ref{R:dense-sobolev}, Schwartz space $S(\ZZ^n)$ is dense in the spaces $H^{(s_j)}(\ZZ^n)$, $j=0,1$, and, hence, in the interpolation space $[H^{(s_0)}(\ZZ^n), H^{(s_1)}(\ZZ^n)]_{\psi}$.) This concludes the proof of the proposition.
\end{proof}

\noindent\textbf{Continuation of the Proof of Theorem~\ref{T:main-1}}

We first prove that (i) implies (ii). Since $\mathscr{S}$ is an interpolation space with respect to the pair $[H^{(s_0)}(\ZZ^n), H^{(s_1)}(\ZZ^n)]$, for some $s_0<s_1$, we may apply proposition~\ref{L-2} to infer $\mathscr{S}=\mathscr{H}_{\psi}$, up to norm equivalence, where $\psi\in\mathcal{B}$ is pseudoconcave in a neighborhood of $\infty$. (By proposition~\ref{L-1}, this means that $\psi$ is an interpolation parameter.) With $\psi\in \mathcal{B}$ at our disposal, define $\varphi$ by the formula~(\ref{E:phi-psi}). Next we apply proposition~\ref{L-3-b} to infer that the function $\varphi$ satisfies the condition~(\ref{E:RO-3}), and, hence, $\varphi\in RO$.  To conclude the proof of the implication, it remains to use proposition~\ref{L-3}.

We now prove that (ii) implies (i). Let $\varphi\in RO$ be a function satisfying the condition~(\ref{E:RO-3}), and define $\psi$ as in~(\ref{E:phi-psi-inverse}), with $s_0<s_1$ as in~(\ref{E:RO-3}). By proposition~\ref{L-4} it follows that $\psi\in\mathcal{B}$ and $\psi$ is an interpolation parameter, and, moreover, $\psi$ satisfies the condition~(\ref{E:phi-psi}). To conclude the proof of the implication, it remains to apply proposition~\ref{L-3}. $\hfill\square$

\section{Proof of Corollary~\ref{C:cor-1}}\label{S:S-4}
The implication (i)$\implies$(ii) follows directly from theorem~\ref{T:main-1}. We now prove the implication (ii)$\implies$(i). Let $\varphi\in RO$ and let $s_0<\sigma_0(\varphi)$ and $s_1>\sigma_1(\varphi)$. Then $\varphi$ satisfies the condition~(\ref{E:RO-3}). Thus, by theorem~\ref{T:main-1}, we have (up to norm equivalence) that $H^{\varphi}(\ZZ^n)$ is an interpolation space with respect to the pair
$[H^{(s_0)}(\ZZ^n), H^{(s_1)}(\ZZ^n)]$. Hence, up to norm equivalence, $H^{\varphi}(\ZZ^n)$ is an interpolation space with respect to the scale $\{H^{(s)}(\ZZ^n)\colon s\in\RR\}$. $\hfill\square$

\section{Proof of Theorem~\ref{T:main-2}}\label{S:S-5}
By proposition~\ref{L-4} we have $\psi\in\mathcal{B}$ and $\psi$ is an interpolation parameter. Furthermore, by the same proposition, $\psi$ satisfies~(\ref{E:phi-psi}). Therefore, theorem~\ref{T:main-2} follows from proposition~\ref{L-3}. $\hfill\square$

\section{Proof of Theorem~\ref{T:main-3}}\label{S:S-6} With theorem~\ref{T:main-2} at our disposal, the proof of theorem~\ref{T:main-3} we follows the scheme of~\cite{MM-15} (see theorem 5.2 there) for the case $\Omega\subset\RR^n$, where $\Omega$ is a bounded domain with Lipschitz boundary.

The following proposition can be found in~\cite{MM-14} (see theorem 1.3 there):

\begin{prop}\label{P-MAIN-3-1} Assume that $\lambda,\eta,\psi\in\mathcal{B}$ and that $\frac{\lambda}{\eta}$ is bounded in a neighbourhood
of $\infty$. Let $\mathscr{H}$ be an admissible pair of Hilbert spaces. Then, with the notations of section~\ref{SS-2-1}, we have

\begin{enumerate}
  \item [(i)] $[\mathscr{H}_{\lambda}, \mathscr{H}_{\eta}]$ is an admissible pair;
  \item [(ii)] we have $[\mathscr{H}_{\lambda}, \mathscr{H}_{\eta}]_{\psi}= \mathscr{H}_{\omega}$, with norm equality, where

\begin{equation}\label{E:eqn-def-greek}
\omega(t) := \lambda(t) \psi\left(\frac{\eta(t)}{\lambda(t)}\right),\qquad t>0.
\end{equation}
\item [(iii)] If $\lambda,\eta,\psi\in\mathcal{B}$ are interpolation parameters, then so is $\omega$.
\end{enumerate}
\end{prop}

The first step in the proof of theorem~\ref{T:main-3} is to choose numbers $s_0,\,s_1$ such that $s_0<\sigma_0(\varphi_j)$ and $s_1>\sigma_1(\varphi_j)$, $j=0,1$. With these $s_0$ and $s_1$, we define $\lambda_j$,  $j=0,1$, as
\begin{equation}\label{E:phi-psi-inverse-sigma}
\lambda_j(t):=\left\{\begin{array}{cc}
                  \tau^{-s_0/(s_1-s_0)}\varphi_j(\tau^{1/(s_1-s_0)}), & \tau\geq 1, \\
                  \varphi_j(1), & 0<\tau<1.
                \end{array}\right.
\end{equation}

Notice that $\frac{\lambda_0}{\lambda_1}$ is bounded in a neighborhood of $\infty$.

By theorem~\ref{T:main-2} we have
\begin{equation}\label{E:temp-3-1}
[[H^{(s_0)}(\ZZ^n), H^{(s_1)}(\ZZ^n)]_{\lambda_0},[H^{(s_0)}(\ZZ^n), H^{(s_1)}(\ZZ^n)]_{\lambda_1}]=[H^{\varphi_0}(\ZZ^n),H^{\varphi_1}(\ZZ^n)].
\end{equation}
As the left hand side displays an admissible pair (see part (i) of proposition~\ref{P-MAIN-3-1}), the pair on the right hand side is also admissible. This proves part (i) of theorem~\ref{T:main-3}.

Let $\psi$ be as in the hypothesis of theorem~\ref{T:main-3}. Returning to~(\ref{E:temp-3-1}) and interpolating with a function parameter $\psi$, we get (after appealing to part (ii) of proposition~\ref{P-MAIN-3-1} with $\mathscr{H}=[H^{(s_0)}(\ZZ^n), H^{(s_1)}(\ZZ^n)]$)

\begin{equation}\label{E:temp-3-2}
[H^{(s_0)}(\ZZ^n), H^{(s_1)}(\ZZ^n)]_{\omega}=[H^{\varphi_0}(\ZZ^n),H^{\varphi_1}(\ZZ^n)]_{\psi},
\end{equation}
where (the interpolation parameter) $\omega$ is given by
\begin{equation}\label{E:phi-psi-inverse-sigma-1}
\omega(t):=\psi_0(t) \psi\left(\frac{\psi_1(t)}{\psi_0(t)}\right),\qquad t\geq 1.
\end{equation}

The same way as it was done in the proof of theorem 5.2 of~\cite{MM-15}, one can check that the function $\varphi$ from~(\ref{E:quad-int}) satisfies the condition
\begin{equation*}
\varphi(t)=t^{s_0}\omega(t^{s_1-s_0}),\quad t\geq 1.
\end{equation*}

Therefore, by proposition~\ref{L-3-b} we have $\varphi\in RO$. Consequently, by proposition~\ref{L-3} we obtain
\begin{equation}\nonumber.
[H^{(s_0)}(\ZZ^n), H^{(s_1)}(\ZZ^n)]_{\omega}=H^{\varphi}(\ZZ^n),
\end{equation}
which, upon taking into account~(\ref{E:temp-3-2}), leads to~(\ref{E:t-3-main}). $\hfill\square$

\section{Proof of Theorem~\ref{T:main-4}}\label{S:S-7} With theorem~\ref{T:main-2} and remark~\ref{R:dense-sobolev} at our disposal, the proofs of properties (i)--(iv) are relatively straightforward. We give the details below:

To prove part (i), start with $\varphi\in RO$ and select numbers $s_0,\,s_1$ such that $s_0<\sigma_0(\varphi)$ and $s_1>\sigma_1(\varphi)$. Defining $\psi$ as in~(\ref{E:phi-psi-inverse}), we have by theorem~\ref{T:main-2} that $\psi\in\mathcal{B}$ is an interpolation parameter and have $[H^{(s_0)}(\ZZ^n), H^{(s_1)}(\ZZ^n)]_{\psi}=H^{\varphi}(\ZZ^n)$, with equality of norms. Referring to part (i) of the definition of the interpolation space (see section~\ref{SS-1-11}), we have a continuous embedding
\begin{equation*}
H^{(s_1)}(\ZZ^n)\hookrightarrow H^{\varphi}(\ZZ^n)\hookrightarrow H^{(s_0)}(\ZZ^n).
\end{equation*}
Now the density of $S(\ZZ^n)$ in $H^{\varphi}(\ZZ^n)$ follows from remark~\ref{R:dense-sobolev} above.

To prove part (ii), first note that hypotheses of theorem~\ref{T:main-3} are satisfied. Thus,
$[H^{\varphi_0}(\ZZ^n),H^{\varphi_1}(\ZZ^n)]$ is an admissible pair, and we have a continuous embedding $H^{\varphi_1}(\ZZ^n)\hookrightarrow H^{\varphi_0}(\ZZ^n)$.

For part (iii), we proceed similarly as in the proof of theorem 4.4 in~\cite{Z-17} for the ``refined Sobolev scale" on a compact manifold. First, according to lemma~\ref{L:pairing-sob} below, for all $s\in\RR$ the sesquilinear form~(\ref{E:inner-l-2}) extends to a sesquilinear duality (separately continuous sesquilinear form)
\begin{equation}\label{E:inner-product-h-phi-reg}
(\cdot,\cdot)\colon H^{(s)}(\ZZ^n)\times H^{(-s)}(\ZZ^n)\to\mathbb{C}.
\end{equation}
The spaces $H^{(s)}(\ZZ^n)$ and $H^{(-s)}(\ZZ^n)$ are dual relative to the duality~(\ref{E:inner-product-h-phi-reg}): for each $s\in \RR$, the map $f(u):=(u,\cdot)$, with $u\in H^{(s)}(\ZZ^n)$, is an isomorphism $f\colon H^{(s)}(\ZZ^n)\to (\overline{H^{(-s)}(\ZZ^n)})'$, where $(\overline{H^{(-s)}(\ZZ^n)})'$ is the anti-dual space of $H^{(-s)}(\ZZ^n)$.

Let $\varphi\in RO$, let $s_0<\sigma_0(\varphi)$, $s_1>\sigma_1(\varphi)$, and let $\psi$ be as in~(\ref{E:phi-psi-inverse}). Thus, we have an isomorphism
\begin{equation}\label{E:4-1-in}
f\colon [H^{(s_0)}(\ZZ^n),H^{(s_1)}(\ZZ^n)]_{\psi}=[(\overline{H^{(-s_0)}(\ZZ^n)})',(\overline{H^{(-s_1)}(\ZZ^n)})']_{\psi}.
\end{equation}
Moreover, by (abstract) theorem 1.4 in~\cite{MM-14}, we have
\begin{equation}\label{E:4-2-in}
[(\overline{H^{(-s_0)}(\ZZ^n)})',(\overline{H^{(-s_1)}(\ZZ^n)})']_{\psi}
=\left(\overline{[(H^{(-s_1)}(\ZZ^n)),(H^{(-s_0)}(\ZZ^n))]_{\widetilde{\psi}}}\right)',
\end{equation}
where $\widetilde{\psi}:=\frac{t}{\psi(t)}$. Appealing again to theorem 1.4 from~\cite{MM-14}, $\widetilde{\psi}\in\mathcal{B}$ is an interpolation parameter.

Since $\psi$ satisfies~(\ref{E:phi-psi}) (as guaranteed by proposition~\ref{L-4}), a quick check shows that
$\widetilde{\varphi}:=\frac{1}{\varphi}$ satisfies
\begin{equation}\nonumber
\widetilde{\varphi}(t)=t^{-s_1}\widetilde{\psi}(t^{-s_0-(-s_1)}).
\end{equation}
Therefore, by proposition~\ref{L-3}, we have (with equality of norms)
\begin{equation}\label{E:4-3-in}
[H^{(-s_1)}(\ZZ^n),H^{(-s_0)}(\ZZ^n)]_{\widetilde{\psi}}=H^{\frac{1}{\varphi}}(\ZZ^n).
\end{equation}
Combining~(\ref{E:4-1-in}),~(\ref{E:4-2-in}),~(\ref{E:4-3-in}) and recalling (see theorem~\ref{T:main-2}) that
\begin{equation}\nonumber
H^{\varphi}(\ZZ^n)=[H^{(s_0)}(\ZZ^n),H^{(s_1)}(\ZZ^n)]_{\psi},
\end{equation}
we get an isomorphism
\begin{equation}\nonumber
f\colon H^{\varphi}(\ZZ^n)\to \left(\overline{H^{\frac{1}{\varphi}}(\ZZ^n)}\right)'.
\end{equation}
Thus, up to norm equivalence, the spaces  $H^{\varphi}(\ZZ^n)$ and $H^{\frac{1}{\varphi}}(\ZZ^n)$ are mutually dual with respect to
the sesquilinear form~(\ref{E:inner-product-h-phi}).

Recall that lemma~\ref{L:pairing-sob} tells us that~(\ref{E:inner-product-h-phi-reg}) is an extension by continuity of the form~(\ref{E:inner-l-2}). Additionally, recall that we have continuous embeddings $H^{(s_1)}(\ZZ^n)\hookrightarrow H^{\varphi}(\ZZ^n)$ and $H^{(-s_0)}(\ZZ^n) \hookrightarrow H^{\frac{1}{\varphi}}(\ZZ^n)$. Therefore,  the form~(\ref{E:inner-product-h-phi}) is an extension by continuity of the form~(\ref{E:inner-l-2}).

To prove part (iv), note that
\begin{align}
&|u(k)|\leq \sum_{k\in\ZZ^n} |u(k)|=\sum_{k\in\ZZ^n}|u(k)|\varphi(\langle k \rangle)[\varphi(\langle k \rangle)]^{-1}\nonumber\\
&\leq \left(\sum_{k\in\ZZ^n}[\varphi(\langle k \rangle)]^{-2}\right)^{1/2}\left(\sum_{k\in\ZZ^n}|u(k)|^2[\varphi(\langle k \rangle)]^2\right)^{1/2}\nonumber\\
&=C\|u\|_{H^{\varphi}},\nonumber
\end{align}
where $0<C<\infty$ is a constant. Here, in the second estimate we used Cauchy--Schwarz inequality, and in the last equality we used the assumption~(\ref{E:sum-phi-hyp}) and the definition~(\ref{E:sob-phi-2}).

Therefore, we get $\|u\|_{\infty}\leq C\|u\|_{H^{\varphi}}$, and this concludes the proof of part (iv). $\hfill\square$

\section{Proof of Theorem~\ref{T:main-5}}\label{S:S-8}
Let $\varphi\in RO$ and let $s_0<\sigma_0(\varphi)$ and $s_1>\sigma_1(\varphi)$. Let $T_{a}=\textrm{Op}[a]$ with $a\in S^{m}(\ZZ^n\times \TT^n)$, where $m\in\RR$. Then, according to corollary 3.3 in~\cite{DK-20} (or corollary 5.6 in~\cite{BRK-20}),
\begin{equation}\label{E:bdd-A-kordyukov}
T_{a}\colon H^{(s_0)}(\ZZ^n)\to H^{(s_0-m)}(\ZZ^n),\qquad T_{a}\colon H^{(s_1)}(\ZZ^n)\to H^{(s_1-m)}(\ZZ^n)
\end{equation}
are bounded linear operators.

Defining $\psi$ as in~(\ref{E:phi-psi-inverse}) and referring to proposition~\ref{L-4} we see that $\psi\in\mathcal{B}$ is an interpolation parameter satisfying~(\ref{E:phi-psi}).  Furthermore, by proposition~\ref{L-3} we have (with equality of norms)
\begin{equation}\label{E:interp-A-1}
[H^{(s_0)}(\ZZ^n), H^{(s_1)}(\ZZ^n)]_{\psi}=H^{\varphi}(\ZZ^n).
\end{equation}

Setting $\widetilde{\varphi}:=t^{-m}\varphi$, we see that~(\ref{E:phi-psi}) can be written as
\begin{equation}\nonumber
\varphi(t)=t^{s_0-m}t^{m}\psi(t^{(s_1-m)-(s_0-m)}),
\end{equation}
that is,
\begin{equation}\nonumber
\widetilde{\varphi}(t)=t^{s_0-m}\psi(t^{(s_1-m)-(s_0-m)}).
\end{equation}

Thus, we may use proposition~\ref{L-3} with $s_0-m$ and $s_1-m$ in place of $s_0$ and $s_1$ respectively and with $\widetilde{\varphi}=t^{-m}\varphi$ in place of $\varphi$. Therefore, we get (with equality of norms)
\begin{equation}\label{E:interp-A-2}
[H^{(s_0-m)}(\ZZ^n), H^{(s_1-m)}(\ZZ^n)]_{\psi}=H^{t^{-m}\varphi}(\ZZ^n).
\end{equation}

Looking at~(\ref{E:bdd-A-kordyukov}), keeping in mind the properties~(\ref{E:interp-A-1})--(\ref{E:interp-A-2}), and remembering the definition of interpolation space (see section~\ref{SS-1-11}), we
get a bounded linear operator
\begin{equation*}
T_{a}\colon H^{\varphi}(\ZZ^n)\to H^{t^{-m}\varphi}(\ZZ^n),
\end{equation*}
and this concludes the proof of the theorem. $\hfill\square$

\section{Proof of Theorem~\ref{T:main-5-f}}\label{S:S-8-f}
Let $\mathscr{B}_{1}$ and $\mathscr{B}_{2}$ be Banach spaces and let $T\colon \mathscr{B}_{1}\to \mathscr{B}_{2}$ be a bounded linear operator.
The notation $(\mathscr{B}_{j})^{*}$ indicates the dual space of $\mathscr{B}_{j}$, $j=1,2$, and the notation $(v,f)_{j}$ describes the action of a functional $f\in (\mathscr{B}_{j})^{*}$ on a vector $v\in \mathscr{B}_{j}$ (Here, the form $(\cdot,\cdot)_{j}$ is linear in the first variable and conjugate linear in the second variable.)

By the adjoint of $T$ we mean a (bounded linear) operator $T^*\colon (\mathscr{B}_{2})^{*}\to(\mathscr{B}_{1})^{*}$ defined as follows:
\begin{equation}\label{E:def-fredh}
(Tu,f)_{2}=(u, T^*f)_{1}, \qquad f\in (\mathscr{B}_{2})^{*},\,\, u\in\mathscr{B}_{1}.
\end{equation}

We now recall an abstract fact (see theorem 2.3.11 in~\cite{agran-90}) regarding Fredholm operators (as defined in section~\ref{SS-1-19-f} above).

\begin{prop}\label{P:agran-90} Let $\mathscr{B}_{1}$ and $\mathscr{B}_{2}$ be Banach spaces and let $T\colon \mathscr{B}_{1}\to \mathscr{B}_{2}$ be a Fredholm operator. Then, $T^*$ is a Fredholm operator.

Furthermore, the following properties hold:
\begin{itemize}
  \item [(i)] $\textrm{Ran } T=\{v\in \mathscr{B}_{2}\colon (v,g)_{2}=0,\,\,\textrm{for all }g\in \textrm{Ker }T^*\}$, where $(\cdot,\cdot)_{2}$ is as in~(\ref{E:def-fredh});
  \item [(ii)] $\textrm{Ker } T^{*}=\{g\in (\mathscr{B}_{2})^{*}\colon (v,g)_{2}=0,\,\,\textrm{for all }v\in \textrm{Ran }T\}$, where $(\cdot,\cdot)_{2}$ is as in~(\ref{E:def-fredh});
  \item [(iii)] the index of $T$ is $\kappa(T)=\dim (\textrm{Ker }T)-\dim (\textrm{Ker }T^*)$.
\end{itemize}
 \end{prop}

The first step in proving theorem~\ref{T:main-5-f} is to establish the Fredholmness of an elliptic operator of order $m$ in the Sobolev scale $H^{(s)}(\ZZ^n)$:

\begin{prop}\label{P:prop-s-s-m} Assume that $a\in ES^{m}(\ZZ^n\times \TT^n)$, where $m\in\RR$ and $ES^{m}(\ZZ^n\times \TT^n)$ is as in definition~\ref{D-2}. Then, for all $s\in\RR$,
\begin{equation}\nonumber
T_{a}\colon H^{(s)}(\ZZ^n)\to H^{(s-m)}(\ZZ^n)
\end{equation}
is a (bounded) Fredholm operator. Furthermore, taking into account the notations~(\ref{E:ker-a}), the following properties hold:
\begin{itemize}
  \item [(i)] $\textrm{Ker } T_{a}=\mathscr{K}$;
  \item [(ii)] $\textrm{Ran } T_{a}=\{v\in H^{(s-m)}(\ZZ^n)\colon (v,w)=0,\,\,\textrm{for all }w\in \mathscr{K}^{\dagger}\}$, where $(\cdot,\cdot)$ is the sesquilinear duality between $H^{(s-m)}(\ZZ^n)$ and $H^{(-s+m)}(\ZZ^n)$, as described in~(\ref{E:inner-product-h-phi-reg}).
  \item [(iii)] the index of $T_{a}$ is $\kappa(T_{a})=\dim (\mathscr{K})-\dim (\mathscr{K}^{\dagger})$. (Hence, $\kappa(T_{a})$ is independent of $s$.)
\end{itemize}
 \end{prop}
\begin{proof} The Fredholmness assertion is justified by using the same argument as in theorem 4.2 of~\cite{DK-20} for the case of $m=s=0$. For completeness, we outline the proof based on an abstract fact due to Atkinson (see propsition 8.2 in~\cite{sh-book}), according to which it is enough to show that we can find a bounded linear operator $B\colon  H^{(s-m)}(\ZZ^n)\to H^{(s)}(\ZZ^n)$ and compact operators $K_1\colon H^{(s)}(\ZZ^n)\to H^{(s)}(\ZZ^n)$ and $K_2\colon H^{(s-m)}(\ZZ^n)\to H^{(s-m)}(\ZZ^n)$ such that
\begin{equation}\label{E:atkinson}
BT_{a}=I+K_{1},\qquad T_{a}B=I+K_{2},
\end{equation}
where $I$ indicates (by slight abuse of notation) the identity operators on $H^{(s)}(\ZZ^n)$ and $H^{(s-m)}(\ZZ^n)$.

Since $T_{a}$ is an elliptic operator of order $m$, theorem 3.6 in~\cite{BRK-20} grants us a parametrix $T_{b}$ with $b\in S^{-m}(\ZZ^n\times \TT^n)$ such that
\begin{equation*}
T_{b}T_{a}=I+K_{1},\qquad T_{a}T_{b}=I+K_{2},
\end{equation*}
where $K_1$ and $K_2$ are operators whose symbols belong to $\cap_{r\in\RR}S^{r}(\ZZ^n\times \TT^n)$.

Thus, $T_{b}$ plays the role of the operator $B$ in~(\ref{E:atkinson}). It remains to show that $K_1\colon H^{(s)}(\ZZ^n)\to H^{(s)}(\ZZ^n)$ and $K_2\colon H^{(s-m)}(\ZZ^n)\to H^{(s-m)}(\ZZ^n)$ are compact operators, which will do just for $K_1$ (as $K_2$ can be handled in the same way).

First, note that for all $t>s$, $K_1\colon H^{(s)}(\ZZ^n)\to H^{(t)}(\ZZ^n)$ is a bounded linear operator (because the symbol of $K_1$ belongs to $\cap_{r\in\RR}S^{r}(\ZZ^n\times \TT^n)$). Furthermore, by theorem 3.5 in~\cite{DK-20}, for all $t>s$ the inclusion
\begin{equation*}
\iota\colon H^{(t)}(\ZZ^n)\to H^{(s)}(\ZZ^n)
\end{equation*}
is a compact operator.

Therefore, writing $K_1=\iota K_1$, we infer that $K_1\colon H^{(s)}(\ZZ^n)\to H^{(s)}(\ZZ^n)$ is a compact operator, which concludes the proof of the Fredholmness property for $T_{a}$.

To prove the property (i), it is enough to observe that for $a\in ES^{m}(\ZZ^n\times \TT^n)$ we can use elliptic regularity to show that $T_{a}u=0$ implies $u\in\cap_{r\in\RR}H^{(r)}(\ZZ^n)=H^{\infty}(\ZZ^n)$. The property (ii) follows from part (i) of proposition~\ref{P:agran-90}, lemma~\ref{L:pairing-sob}, and the fact that $(T_{a})^{\dagger}u=0$ implies $u\in\cap_{r\in\RR}H^{(r)}(\ZZ^n)=H^{\infty}(\ZZ^n)$ (because $(T_{a})^{\dagger}$ is also an elliptic operator of order $m$). With this information, the property (iii) follows from part (iii) of proposition~\ref{P:agran-90}.
\end{proof}

Before proving theorem~\ref{T:main-5-f}, we recall the following interpolation property (see theorem 1.7 in~\cite{MM-14}), with the terminology as in sections~\ref{SS-2-1} and~\ref{SS-1-11}.

\begin{prop}\label{P:murach-abs-f} Let $\mathscr{H}:=[\mathscr{H}_{0},\mathscr{H}_{1}]$ and $\mathscr{G}:=[\mathscr{G}_{0},\mathscr{G}_{1}]$
be two admissible pairs of Hilbert spaces, and let $\psi\in\mathcal {B}$ be
an interpolation parameter, where the class $\mathcal {B}$ is as in section~\ref{SS-2-1}. Assume that $T$ is a linear operator acting on $\mathscr{H}_{0}$ and satisfying the following properties:
\begin{itemize}
  \item [(i)] the restrictions of $T$ to the spaces $\mathscr{H}_{j}$, with  $j=0,1$, are bounded Fredholm operators
  $T\colon \mathscr{H}_{j}\to \mathscr{G}_{j}$, $j=0,1$;
  \item [(ii)] the operators $T\colon \mathscr{H}_{j}\to \mathscr{G}_{j}$, $j=0,1$, have a common kernel;
  \item [(iii)] the operators $T\colon \mathscr{H}_{j}\to \mathscr{G}_{j}$, $j=0,1$, have the same index.
\end{itemize}
Then, the restriction of $T$ to the interpolation space $\mathscr{H}_{\psi}$ is a bounded Fredholm operator $T\colon \mathscr{H}_{\psi}\to \mathscr{G}_{\psi}$ whose kernel and the index are the same as those of the operators $T\colon \mathscr{H}_{j}\to \mathscr{G}_{j}$, $j=0,1$.

Furthermore, we have $T(\mathscr{H}_{\psi})=\mathscr{G}_{\psi}\cap T(\mathscr{H}_{0})$.
\end{prop}

\noindent\textbf{Continuation of the Proof of Theorem~\ref{T:main-5-f}}
\\\\
Let $\varphi\in RO$ and let $s_0<\sigma_0(\varphi)$ and $s_1>\sigma_1(\varphi)$. Since $a\in ES^{m}(\ZZ^n\times \TT^n)$, $m\in\RR$, we can use proposition~\ref{P:prop-s-s-m} to infer that
\begin{equation}\label{E:bdd-A-kordyukov}
T_{a}\colon H^{(s_0)}(\ZZ^n)\to H^{(s_0-m)}(\ZZ^n),\qquad T_{a}\colon H^{(s_1)}(\ZZ^n)\to H^{(s_1-m)}(\ZZ^n)
\end{equation}
are bounded Fredholm operators, with a common kernel and index (as described in items (i) and (iii) of proposition~\ref{P:prop-s-s-m}).

Defining $\psi$ as in~(\ref{E:phi-psi-inverse}) and referring to proposition~\ref{L-4} we see that $\psi\in\mathcal{B}$ is an interpolation parameter satisfying~(\ref{E:phi-psi}). Thus, according to~(\ref{E:interp-A-1}) and~(\ref{E:interp-A-2}) we have
\begin{equation*}
[H^{(s_0)}(\ZZ^n), H^{(s_1)}(\ZZ^n)]_{\psi}=H^{\varphi}(\ZZ^n)
\end{equation*}
and
\begin{equation*}
[H^{(s_0-m)}(\ZZ^n), H^{(s_1-m)}(\ZZ^n)]_{\psi}=H^{t^{-m}\varphi}(\ZZ^n).
\end{equation*}
Thus, we can apply proposition~\ref{P:murach-abs-f} to infer that
\begin{equation*}
T_{a}\colon H^{\varphi}(\ZZ^n)\to H^{t^{-m}\varphi}(\ZZ^n)
\end{equation*}
is a bounded Fredholm operator whose kernel and index are described in items (i) and (iii) of proposition~\ref{P:prop-s-s-m}. This proves the Fredholmness property and the assertions (i) and (iii) of the theorem.

We now justify assertion (ii) of the theorem. By the last sentence in proposition~\ref{P:murach-abs-f}, we have
\begin{equation*}
T_{a}(H^{\varphi}(\ZZ^n))= H^{t^{-m}\varphi}(\ZZ^n)\cap T(H^{(s_0)}(\ZZ^n)).
\end{equation*}
Combining this equality with item (ii) of proposition~\ref{P:prop-s-s-m} (with $s_0$ in place of $s$) and using the fact that $H^{t^{-m}\varphi}(\ZZ^n)\subset H^{(s_0-m)}(\ZZ^n)$, we obtain
\begin{equation*}
\textrm{Ran } T_{a}=\{v\in H^{t^{-m}\varphi}(\ZZ^n)\colon (v,w)=0,\,\,\textrm{for all }w\in \mathscr{K^{\dagger}}\},
\end{equation*}
where $\mathscr{K^{\dagger}}$ is as in~(\ref{E:ker-a}). This concludes the proof of the theorem $\hfill\square$

\section{Proof of Theorem~\ref{T:main-7}}\label{S:S-9} The following abstract result was established in theorem 2.5 of~\cite{MM-21}:

\begin{prop}\label{P:6-1}Let $\varphi\in RO$, $s_0<\sigma_0(\varphi)$, and $s_1>\sigma_1(\varphi)$. Let $\psi$ be as in~(\ref{E:phi-psi-inverse}). Let $H^{(s_j)}_{A}$, $j=0,1$, and $H^{\varphi}_{A}$ be as in section~\ref{SS-1-21}.

Then, $\psi\in\mathcal{B}$ and $\psi$ is an interpolation parameter. Furthermore, we have (with equality of norms)
\begin{equation*}
[H^{(s_0)}_{A}, H^{(s_1)}_{A}]_{\psi}=H^{\varphi}_{A}.
\end{equation*}
\end{prop}

Let us go back to the setting of theorem~\ref{T:main-7}, as described in section~\ref{SS-1-22}.

\begin{lemma}\label{L:A-k-isomorphic} Assume that $A:=\textrm{Op}[a]$ satisfies the hypotheses (H1)--(H2). Then, up to norm equivalence, we have
\begin{equation*}
H_{A}^{(k)}(\ZZ^n)=H^{(k)}(\ZZ^n),\quad \textrm{for all }k\in\mathbb{Z}.
\end{equation*}
\end{lemma}
\begin{proof}
It is enough to show the result for $k\in\NN_{0}$, as the case  $k=-1,-2,\dots$ follows by duality from the case $k\in\NN$.
The case $k=0$ is obviously true. Therefore, we need to consider the case $k\in\NN$ only.

Remembering the hypotheses (H1)--(H2) and remark~\ref{R:kor}, we see that $A$ is a positive self-adjoint operator in $\ell^2(\ZZ^n)$ such that $\dom(A)=H^{(1)}(\ZZ^n)$. Therefore $A$ establishes an isomorphism $A\colon H^{(1)}(\ZZ^n)\to \ell^2(\ZZ^n)$. Taking into account the definition of $H^{(1)}_{A}(\ZZ^n)$, we infer that $H^{(1)}_{A}(\ZZ^n)=H^{(1)}(\ZZ^n)$, up to norm equivalence.

As in the case of pseudo-differential operators on $\RR^n$ (with symbols in the uniform H\"ormander class), we have the following ``elliptic regularity" property for $A=\textrm{Op}[a]$ with $a\in ES^{1}(\ZZ^n\times \TT^n)$: If $s>1$ and if $u\in H^{(1)}(\ZZ^n)$ satisfies $Au\in H^{(s-1)}(\ZZ^n)$, then $u\in H^{(s)}(\ZZ^n)$. (Theorem 3.6 in~\cite{BRK-20} guarantees the existence of (a unique) parametrix $B=\textrm{Op}[b]$ with $b\in S^{-1}(\ZZ^n\times \TT^n)$. Then, the mentioned ``elliptic regularity" property follows by using the usual parametrix-type argument.)

Using the ``elliptic regularity" property and the mentioned isomorphism
\begin{equation*}
A\colon H^{(1)}(\ZZ^n)\to \ell^2(\ZZ^n),
\end{equation*}
we infer that $A^k$ gives rise to an isomorphism
\begin{equation}\label{E:iso-temp-5}
A^k\colon H^{(k)}(\ZZ^n)\to \ell^2(\ZZ^n).
\end{equation}
Using the isomorphism~(\ref{E:iso-temp-5}) and the definition of $H^{(k)}_{A}(\ZZ^n)$, we arrive at $H^{(k)}_{A}(\ZZ^n)=H^{(k)}(\ZZ^n)$, up to norm equivalence.
\end{proof}

\noindent\textbf{Continuation of the Proof of Theorem~\ref{T:main-7}}
\\\\
Let $\varphi\in RO$, let $k\in\NN$ be a number such that $-k<\sigma_0(\varphi)$ and $k>\sigma_1(\varphi)$. Define $\psi$ as in~(\ref{E:phi-psi-inverse}) with $s_0=-k$ and $s_1=k$. Proposition~\ref{P:6-1} tells us that (up to norm equivalence)
\begin{equation}\label{E:m-7-1}
[H^{(-k)}_{A}, H^{(k)}_{A}(\ZZ^n)]_{\psi}=H^{\varphi}_{A}(\ZZ^n).
\end{equation}

Taking into account~(\ref{E:m-7-1}) and lemma~\ref{L:A-k-isomorphic} we obtain, up to norm equivalence,
\begin{equation*}
[H^{(-k)}(\ZZ^n), H^{(k)}(\ZZ^n)]_{\psi}=H^{\varphi}_{A}(\ZZ^n).
\end{equation*}
On the other hand, theorem~\ref{T:main-2} tells us that
\begin{equation*}
H^{\varphi}(\ZZ^n)=[H^{(-k)}(\ZZ^n), H^{(k)}(\ZZ^n)]_{\psi}.
\end{equation*}
Therefore, we have (up to norm equivalence), $H^{\varphi}_{A}(\ZZ^n)=H^{\varphi}(\ZZ^n)$. $\hfill\square$

\appendix
\section{Anti-duality between discrete Sobolev spaces}\label{S:App}
In this section we adapt the proof of lemma 4.4.4 in~\cite{BC-09} to discrete Sobolev spaces.

We begin with a brief review of terminology. Denote by $\overline{V}$ the \emph{conjugate} of a (complex) vector space $V$. Viewed as real vector spaces, $\overline{V}$ and $V$ are the same. The only difference is that the multiplication of a vector $v\in \overline{V}$ by a scalar $z\in\CC$ is defined as $\overline{z}v$, where $\overline{z}$ is the conjugate of $z$. The space of linear (respectively, anti-linear) functionals on $V$ will be denoted by $V^{'}$ (respectively, $(\overline{V})^{'}$).

Let $V_1$ and $V_2$ be two topological (complex) vector spaces. By a  \emph{pairing}  of $V_1$ and $V_2$ we mean a continuous sesquilinear  map $B\colon V_1\times V_2\to \CC$. (Here, ``sesquilinear" means linear in the first and anti-linear in the second slot.)  The pairing $B$ gives rise to continuous linear maps $\tau\colon V_1\to (\overline{V_2})^{'}$ and $\kappa\colon \overline{V_2}\to V_{1}^{'}$ as follows:
\begin{equation*}
\tau (u):=B(u,\cdot)\,\qquad \kappa(v):=B(\cdot, v).
\end{equation*}
The pairing $B$ is said to be \emph{perfect} if $\tau$ and $\kappa$ are (linear) isomorphisms.

\begin{lemma}\label{L:pairing-sob} Let $s\in\RR$. Then the $\ell^2$-inner product~(\ref{E:inner-l-2}), initially considered on $\mathcal{S}(\ZZ^n)\times \mathcal{S}(\ZZ^n)$, extends to a (continuous sesquilinear) pairing
\begin{equation}\label{E:inner-product-h-phi-reg}
(\cdot,\cdot)\colon H^{(s)}(\ZZ^n)\times H^{(-s)}(\ZZ^n)\to \mathbb{C}.
\end{equation}
The pairing $(\cdot,\cdot)$ is perfect and induces isometric (linear) isomorphisms $\overline{H^{(-s)}(\ZZ^n)}\cong(H^{(s)}(\ZZ^n))^{'}$ and $H^{(s)}(\ZZ^n)\cong(\overline{H^{(-s)}(\ZZ^n)})^{'}$.
\end{lemma}
\begin{proof}
Let $(\cdot,\cdot)$ be as in~(\ref{E:inner-l-2}) and let $\langle k\rangle$ be as in~(\ref{sob-mult}). For all $u,v\in \mathcal{S}(\ZZ^n)$ we have
\begin{equation}\label{E:a-1}
|(u,v)|=|(\langle k\rangle^s u,\langle k\rangle^{-s} v)|\leq \|\langle k\rangle^s u\|\|\langle k\rangle^{-s}v\|=\|u\|_{H^{(s)}}\|v\|_{H^{(-s)}},
\end{equation}
where we used Cauchy--Schwarz inequality in $\ell^2(\ZZ^n)$.

As $\mathcal{S}(\ZZ^n)$ is dense in the spaces $H^{\pm s}(\ZZ^n)$ (see lemma 3.16 in~\cite{DK-20}), from~(\ref{E:a-1}) we infer that the $\ell^2$-inner product extends to a (continuous sesquilinear) pairing $(\cdot,\cdot)\colon H^{(s)}(\ZZ^n)\times H^{(-s)}(\ZZ^n)\to \mathbb{C}$.

Additionally, from~(\ref{E:a-1}) we obtain
\begin{equation}\label{E:a-2}
\|u\|_{H^{(s)}}=\sup\{|(u,v)|\colon v\in \mathcal{S}(\ZZ^n) \textrm{ and } \|v\|_{H^{(-s)}}=1\}.
\end{equation}
Using the density of $\mathcal{S}(\ZZ^n)$ in the space $H^{(-s)}(\ZZ^n)$, observe that the right hand side of~(\ref{E:a-2}) is equal to the norm of the functional $\tau(u)\in(\overline{H^{(-s)}(\ZZ^n)})^{'}$ corresponding to $u\in H^{(s)}(\ZZ^n)$ via $\tau(u):=(u,\cdot)$.  Therefore, $\tau\colon H^{(s)}(\ZZ^n)\to  (\overline{H^{(-s)}(\ZZ^n)})^{'}$ is an isometry.

In the same way, using~(\ref{E:a-1}) we can show that the map $\kappa\colon \overline{H^{(-s)}(\ZZ^n)}\to (H^{(s)}(\ZZ^n))^{'}$ defined as
$\kappa(v):=(\cdot,v)$ is an isometry.

In particular, the maps $\tau$ and $\kappa$ are injective. From the injectivity of $\kappa$ we infer that the range of $\tau$ is dense in $(\overline{H^{(-s)}(\ZZ^n)})^{'}$. As $\tau$ is an isometry, it follows that $\tau$ is surjective. Therefore, $\tau$ is an isometric isomorphism. In the same way, we can show that $\kappa$ is an isometric isomorphism.
\end{proof}

\end{document}